\documentclass[10pt]{amsart}
\usepackage{amsmath,amssymb,amsfonts}

\newtheorem{theorem}{Theorem}
\newtheorem{lemma}{Lemma}
\newtheorem{proposition}{Proposition}
\newtheorem{definition}{Definition}
\newtheorem*{corollary}{Corollary}
\newtheorem{remark}{Remark}

\begin{document}

\title{On sequences of Toeplitz matrices over finite fields}
\author{Geoffrey Price* and Myles Wortham}
\address{\noindent Department of Mathematics, United States Naval Academy, Annapolis MD 21402\\
Naval Base San Diego, San Diego CA  92136}
\email{glp@usna.edu}
\thanks{* supported in part by the United States Naval Research Laboratory, Washington, DC}
\date{3 September 2019}
\begin{abstract}
For each non-negative integer $n$ let $\mathcal{A}_n$ be an $n+1$ by $n+1$ Toeplitz matrix over a finite field, $F$, and suppose for each $n$ that $\mathcal{A}_n$ is embedded in the upper left corner of $\mathcal{A}_{n+1}$.  We study the structure of the sequence $\nu = \{\nu_n :n \in \mathbb{Z}^+\}$, where $\nu_n = \text{null}(\mathcal{A}_n)$ is the nullity of $\mathcal{A}_{n}$.  For each $n\in \mathbb{Z}^+$ and each nullity pattern
$\nu_0,\nu_1,\dots,\nu_n$, we count the number of strings of Toeplitz matrices $\mathcal{A}_0,\mathcal{A}_1,\dots,\mathcal{A}_{n}$ with this pattern.  As an application we present an elementary proof of a result of D. E. Daykin on the number of $n\times n$ Toeplitz matrices over $GF(2)$ of any specified rank. (This is a corrected version of the paper published in Linear Algebra and Its Applications, $\bold{561}$ \, $(2019), 63-80$.)\\
2000 MSC Classification  15A33, 15A57\\
Keywords:  Toeplitz matrix, nullity sequence, rank, finite fields
\end{abstract}
\maketitle

\section{Introduction}\label{S: Intro}

In this paper we consider sequences $\{ \mathcal{A}_n \}$ of Toeplitz matrices with entries in an arbitrary finite field $F$.  The sequence $\{ \mathcal{A}_n \}$ is uniquely determined by a pair of sequences $a_0,a_1,\dots $ and $b_1,b_2,\dots$ over $F$, so that $\mathcal{A}_n$ is the $(n+1)\times(n+1)$ matrix
\begin{equation*}
\left[
\begin{matrix}
a_0 &  a_1 & a_2 & \dots &  a_{n-1}&a_n\\
b_{1} & a_0 & a_1 & \dots &  a_{n-2}&a_{n-1}\\
b_{2} & b_{1} & a_0 & \dots &  a_{n-3}&a_{n-2}\\
\vdots & \vdots & \vdots & \ddots &  \vdots & \vdots \\
b_{n} & b_{n-1} & b_{n-2} & \dots & b_{1} &a_0
\end{matrix}
\right]\tag{1.1}
\end{equation*}

\noindent In particular, $\mathcal{A}_n$ lies in the top left corner (as well as the bottom right corner) of $\mathcal{A}_{n+1}$.

For each sequence $\{ \mathcal{A}_n \}$ let $\nu =\{\nu_n\}$ be the corresponding sequence of nullities, i.e. $\nu_n$ is the nullity of $\mathcal{A}_n$.  As in \cite{CP}, where K. Culler and one of the authors considered nullity sequences of skew centro-symmetric matrices $\mathcal{A}_n$ (see also \cite{PP},\cite{PT}), here we ascertain the patterns of the nullity sequences which can occur for Toeplitz matrices.  We present an elementary proof which shows that $\{ \nu_n \}$ is a concatenation of strings of the form $0,0,\dots,0$ or $1,2,\dots,d-1,d,d,\dots,d,d-1,d-2,\dots,1,0$, where $d$ can repeat any finite number of times; or the nullity sequence may consist of a concatenation of finitely many strings of this type, followed by $1,2,\dots$, see Theorem \ref{T: null_string}   . These patterns were observed in \cite{HR84} for sequences of Toeplitz matrices over $\mathbb{C}$. We analyze the structure of the (right) kernels of matrices which satisfy these patterns and use this to determine the number of matrices $\mathcal{A}_0,\dots,\mathcal{A}_{n}$ which satisfy a specific nullity pattern, $\nu_0,\dots,\nu_{n}$.  We note that the structure of kernels of Toeplitz matrices over $\mathbb{C}$ is carried out in \cite{HR84}, and a similar analysis is carried out for kernels of centro-symmetric and skew centro-symmetric matrices over finite fields (of characteristic not equal to $2$) in \cite{HR02}.  We believe that our results on the number of strings of Toeplitz matrices $\{\mathcal{A}_n \}$ are new, however.  As an application of our results we present a straightforward argument which
determines the number of $n\times n$ Toeplitz matrices over $GF(2)$ of any specified rank, a result obtained earlier over arbitrary finite fields by Daykin, \cite{Day}.

For results on extensions $\mathcal{A}_{n+1}$ of Toeplitz matrices $\mathcal{A}_n$ with the same rank over the complex numbers, see \cite{I68}.  See also \cite{I82} Section $13$ for a determination of the singular extensions $\mathcal{A}_{n+1}$ of non-singular Toeplitz matrices $\mathcal{A}_n$ over $\mathbb{C}$.

\section{Nullity Sequences}\label{S: Nullity}

The main goal of this section is to determine the structure of any nullity sequence for a sequence of Toeplitz matrices over a finite field $F$, defined as in the previous section.  These results are assembled in Theorem \ref{T: null_string} below, and are obtained by making an analysis of the sequence $\{ \text{ker}(\mathcal{A}_n) \}$ of kernels of the matrices.  We shall show, for example, for a finite section, or string, $\nu_{\ell},\dots,\nu_{\ell +m}$ of the nullity sequence, taking values  $1,2,\dots,d-1,d,d,\dots,d,d-1,\dots,2,1,0$, the kernels of the matrices $A_{\ell +1},\dots,A_{\ell + m}$ can all be described explicitly in terms of the single vector generating ker$(\mathcal{A}_{\ell})$, cf. \cite{HR84} for related results on Toeplitz matrices over $\mathbb{C}$. We shall also be concerned with computing, starting with a fixed matrix $\mathcal{A}_{\ell}$ of nullity $\nu_{\ell} = 1$, the number of strings of embedded matrices $\mathcal{A}_{\ell}$ through $\mathcal{A}_{\ell + m}$ having a prescribed nullity string of the types listed in Theorem \ref{T: null_string}. This analysis begins in this section with Theorems \ref{T: null_one_vec} through \ref{T: equal equal down}, and their corollaries, and continues into the next section.

References \cite{HR84} and \cite{ER13} (and many other references) also contain detailed information about the kernels of Toeplitz matrices over $\mathbb{C}$ and over arbitrary fields, respectively.  The approach taken to describe the kernel of a Toeplitz matrix in these and other papers leads to some of the same results that we obtain from our approach.  In particular, the statement of our Theorem \ref{T: null_string}, is equivalent to that of Proposition 5.12 in \cite{HR84}, except that the latter addresses Toeplitz matrices with entries in $\mathbb{C}$.  On the other hand we believe that our approach is useful in obtaining our main results, which we believe to be new.  For example, our Theorem \ref{T: null_one_vec} begins the argument which leads not only to the structure theorem for kernels of strings of Toeplitz matrices, Theorem \ref{T: null_string}, but to our main results, contained in Remark \ref{R: counting_summary}, which give complete information on the numbers of strings of Toeplitz matrices which satisfy the patterns arising in Theorem \ref{T: null_string}.

In what follows we shall assume that $\vert F \vert = q$.

\begin{definition}\label{D: null_num}
$N(n,\nu)$ is the number of $(n+1)\times(n+1)$ Toeplitz matrices $\mathcal{A}_n$ as in (1.1) with nullity $\nu$.
\end{definition}

For each $n\in \mathbb{N}$ the matrix $\mathcal{A}_{n-1}$ occupies the top left $n\times n$ corner of $\mathcal{A}_{n}$, and also the lower right corner of $\mathcal{A}_{n}$.  From this we conclude that the sequence of ranks $\{\rho_n = \text{rank}(\mathcal{A}_n): n\in \mathbb{Z}^+\}$ is monotone increasing.  This and the Rank-Nullity Theorem give the following result.

\begin{proposition}\label{P: null} For each $n\in \mathbb{Z}^+$, $\rho_{n+1}$ is one of $\rho_n$, $1+ \rho_n,$ or $2+\rho_n$ so $\nu_{n+1}$ is $\nu_n+1$, $\nu_n$, or $\nu_{n}-1$.
\end{proposition}

Let $a_0,a_1,\dots $ and $b_1,b_2,\dots$ be a pair of sequences in $F$ and let $\{\mathcal{A}_n:n\in \mathbb{Z}^+\}$ be the sequence of Toeplitz matrices defined as above.

\begin{definition}\label{D: null_seq}  Given a fixed sequence of Toeplitz matrices defined from a pair of sequences in $F$, the corresponding {\it nullity sequence}, $\{\nu_n: n\in \mathbb{Z}^+\}$ is given by $\nu_n = \text{null}(\mathcal{A}_n) = dim(\text{ker}(\mathcal{A}_n))$.
\end{definition}

We will prove the following result on the possible forms of the nullity sequence (see Remark \ref{R: null_string}).
\begin{theorem}\label{T: null_string} (cf. \cite{HR84}, Proposition 5.12)
Every nullity sequence  $\nu = \{ \nu_n:n\in \mathbb{Z}^+\}$ is given as follows.
\begin{enumerate}
\item $\nu = \{1,2,\dots,d-1,d,d,d\dots\}$, for some $d\in \mathbb{N}$
\item $\nu = \{0,1,2,\dots\}$
\item $\nu = \{1,2,3,\dots\}$
\item $\nu$ is the concatenation of finite length strings of the following form.
    \begin{enumerate}
         \item an arbitrary number of zeroes
\item $1,2,\dots,d-1,d,d-1,\dots,2,1,0 \,\, \text{for some positive integer}\,\, d \geq 1$
\item $1,2,\dots,d-1,d,d,\dots,d,d-1,\dots,2,1,0,$ where $d\geq 1$ may appear an arbitrary number of times
 \end{enumerate}
\item $\nu$ is the concatenation of a finite number of finite length strings of the form above, followed by (1) or (3)
\end{enumerate}
\end{theorem}

Let $n\in \mathbb{N}$ and suppose $\mathcal{A}_{n-1}$ and $\mathcal{A}_n$ are given matrices with $\mathcal{A}_{n-1}$ embedded in $\mathcal{A}_n$ (in both the top left and lower right corners) and satisfying $\nu_{n-1} =0$, i.e. $\mathcal{A}_{n-1}$ is invertible, and $\nu_n = 1$.  Then we will show that for any positive integer $d$ there is one and only one string of matrices beginning with $\mathcal{A}_{n}$, written $\mathcal{A}_n,\dots ,\mathcal{A}_{n+d}$ whose nullities $\nu_n,\dots,\nu_{n+d}$ are $1,2,\dots,d+1$, respectively, see the Corollary of Theorem \ref{T: ker_form}.  First it will be helpful to introduce the following notation and terminology.

\begin{itemize}
\item $F^n$ will denote the vector space of column vectors of length $n$ over $F$.
\item  The superscript $t$ denotes transpose so $\bold k= [k_0,\dots,k_n]^t$ is a column vector.
\item  If $\bold k= [k_0,\dots,k_n]^t$, then $\bold k'$ is the vector obtained from $\bold k$ by deleting its last entry.
\item  If $\bold k= [k_0,\dots,k_n]^t$, then $'\bold k$ is the vector obtained from $\bold k$ by deleting its initial entry.
\item  $\omega(\bold k) = [k_0,\dots,k_n,0]^t$, i.e., $\omega(\bold k)$ appends a $0$ to the end of vector $\bold k$.
\item  $\sigma(\bold k) = [0,k_0,\dots,k_n]^t$, i.e., $\sigma(\bold k)$ places a $0$ at the beginning of $\bold k$.
\item We will say that $\mathcal{A}_n$ is embedded in $\mathcal{A}_{n+1}$ if $\mathcal{A}_n$ is the top left corner (also lower right corner) of $\mathcal{A}_{n+1}$, and we will denote this by $\mathcal{A}_n \rightharpoonup \mathcal{A}_{n+1}$ or $\mathcal{A}_{n+1} \leftharpoonup \mathcal{A}_n$.
\end{itemize}

\begin{theorem}\label{T: null_one_vec}  Let $\mathcal{A}_{n-1}$ and $\mathcal{A}_n$ be Toeplitz matrices with $\mathcal{A}_{n-1}\rightharpoonup \mathcal{A}_n$, and nullity $\nu_{n-1} = 0, \nu_{n}=1$, respectively.  Then the column vector $\bold k = [k_0,\dots,k_n]^t$ spanning $\text{ker}(\mathcal{A}_n)$ satisfies $k_0 \neq 0$ and $k_n \neq 0$.
\end{theorem}
\begin{proof}
Let $\bold k$ be the vector spanning $\text{ker}(\mathcal{A}_n)$.  If $k_0 = 0$ then the vector $'\bold k$ is in the kernel of $\mathcal{A}_{n-1}$ in the lower right corner of $\mathcal{A}_n$, a contradiction unless $'\bold k = \bold 0$. Then $k_0 \neq 0$.  Next suppose that $k_n = 0$.  Since $\nu_n = 1$ we have $\rho_{n-1} = \rho_n$ so the last column of $\mathcal{A}_n$ is in the span of its preceding columns.  By symmetry of $\mathcal{A}_n$ this implies that the first row of $\mathcal{A}_n$ is in the span of the succeeding rows.  Therefore $\bold k'$ is in the kernel of the copy of $\mathcal{A}_{n-1}$ embedded in the top left of $\mathcal{A}_n$, a contradiction.  Therefore $k_n \neq 0$, also.
\end{proof}

\begin{theorem}\label{T: ker_form}  Let $\mathcal{A}_{n-1},\mathcal{A}_n,\dots,\mathcal{A}_{n+d}$ be matrices embedded in one another with nullities $\nu_{n-1},\nu_n,\dots,\nu_{n+d}$ equal to $0,1,2,\dots,d+1$, respectively.  Then for $1 \leq j \leq d-1, \text{ker}(\mathcal{A}_{n+j})$ coincides with span$(\omega(\text{ker}(\mathcal{A}_{n+j-1})),\sigma(\text{ker}(\mathcal{A}_{n+j-1}))$.
\end{theorem}
\begin{proof}
As in the previous proof we have $\rho_{n+j-1} = \rho_{n+j}$ and so we conclude as we did in the previous proof that the last row (respectively, the first row) of $\mathcal{A}_{n+j}$ is in the span of the preceding (resp., the succeeding rows) of $\mathcal{A}_{n+j}$.  Since $\omega(\text{ker}(\mathcal{A}_{n+j-1}))$ is in the kernel of the matrix obtained by deleting the last row of $\mathcal{A}_{n+j}$, and since the last row of $\mathcal{A}_{n+j}$ is in the span of the preceding rows, we conclude that $\omega(\text{ker}(\mathcal{A}_{n+j-1}))\subset \text{ker}(\mathcal{A}_{n+j})$.  Similarly $\sigma(\text{ker}(\mathcal{A}_{n+j-1}))\subset \text{ker}(\mathcal{A}_{n+j})$.

From the results of the previous paragraph we see inductively that for $1\leq j \leq d$, $\text{ker}(\mathcal{A}_{n+j})$ is spanned by $[k_0,\dots,k_n,0,\dots,0]^t =
\omega^j(\bold k), [0,k_0,\cdots,k_n,0,\cdots,0]^t = \sigma(\omega^{j-1}(\bold k)),\dots, [0,\dots,0,k_0,\dots,k_n]^t = \sigma^j(\bold k)$.  The first $n+j$ (out of $n+j+1$) of these vectors is in $\omega(\text{ker}(\mathcal{A}_{n+j-1}))$ and the last $n+j$ vectors are in $\sigma(\mathcal{A}_{n+j-1})$.
\end{proof}
\begin{corollary}
Assume the same hypotheses as in the previous theorem.  Then for each $j$, $1\leq j \leq d$ there is one and only one choice of a pair $\{b_{n+j},a_{n+j}\}$ such that $\nu_{n+j} = null(\mathcal{A}_{n+j}) = 1 + \nu_{n+j-1} = 1 + null(\mathcal{A}_{n+j-1})$.
\end{corollary}
\begin{proof}
From Theorem \ref{T: null_one_vec} the vector $\bold k = [k_0,\dots,k_n]^t \in \text{ker}(\mathcal{A}_n)$ satisfies $k_0 \neq 0$ and $k_n \neq 0$.  The vector $\omega^j(\bold k)$ is in the kernel of $\mathcal{A}_{n+j}$ and so it has dot product $0$ with the last row of $\mathcal{A}_{n+j}$.  Since the first entry of this vector is $k_0$ we conclude that $b_{n+j}$ is uniquely determined by this fact.  Similarly $\sigma^j(\bold k)\in \text{ker}(\mathcal{A}_{n+j})$ and the vector ends in $k_n$, so $a_{n+j}$ is uniquely determined.  Thus there is only one choice for the pair $\{b_{n+j},a_{n+j}\}$.
\end{proof}
Our results imply that if $\nu_{n-1},\dots,\nu_{n+d}$ is $0,1,\dots,d+1$ then there is one and only one choice of the pair $\{b_{n+d+1},a_{n+d+1}\}$ such that $\nu_{n+d+1} = d+2$.  We now show that there are $(q-1)^2$ choices of $\{b_{n+d+1},a_{n+d+1}\}$ so that $\nu_{n+d+1} = d$.
We also determine how to obtain $\text{ker}(\mathcal{A}_{n+d+1})$ from $\text{ker}(\mathcal{A}_{n+d})$.

\begin{theorem}\label{T: null_up}
Suppose $n\in \mathbb{N}$ is such that $\mathcal{A}_{n-1},\mathcal{A}_n,\dots,\mathcal{A}_{n+d}$ is an embedded string of Toeplitz matrices with nullities $0,1,\dots,d+1$.  Then there are $(q-1)^2$ choices of $\{b_{n+d+1},a_{n+d+1}\}$ so that $\nu_{n+d+1} = d$.
\end{theorem}
\begin{proof}
Consider the matrix
\begin{equation*}
\mathcal{A}_{n+d+1} =\left[
\begin{matrix}
a_0 &  a_1 & a_2 & \dots &  a_{n+d}&a_{n+d+1}\\
b_1 & a_0 & a_1 & \dots &  a_{n+d-1}&a_{n+d}\\
b_2 & b_1 & a_0 & \dots &  a_{n+d-2} &a_{n-d-1}\\
\vdots & \vdots & \vdots & \ddots &  \vdots & \vdots \\
b_{n+d} & b_{n+d-1} & b_{n+d-2} & \dots & a_{0} &a_1\\
b_{n+d+1} & b_{n+d} & b_{n+d-1} & \dots & b_{1} & a_0
\end{matrix}
\right]
\end{equation*}
As usual we view the smaller matrix $\mathcal{A}_{n+d}$ positioned in both the top left and bottom right corner of $\mathcal{A}_{n+d+1}$.

From the proof of Theorem \ref{T: ker_form}, the kernel of  $\mathcal{A}_{n+d}$ has a basis $[k_0,\dots,k_n,0,\dots,0]^t$, $[0,k_0,\dots,k_n,0,\dots,0]^t,  \dots,[0,\dots,0,k_0,\dots,k_n]$ in $F^{n+d+1}$.  Label these vectors $\bold v_0,$
$\bold v_1,\dots,\bold v_{d}$.  Next append a $0$ to the end of each of these vectors, obtaining the vectors $\omega(\bold v_0),\omega(\bold v_1),\dots,
\omega(\bold v_d)$. Note that each of these vectors has dot product $0$ with the first $n+d+1$ rows of $\mathcal{A}_{n+d+1}$, since $\bold v_0$ through $\bold v_d$ form a basis for $\text{ker}(\mathcal{A}_{n+d})$.  Note, however, that the vectors $\omega(\bold v_1),\dots, \omega(\bold v_d)$ all have dot product $0$ with the last row of $\mathcal{A}_{n+d+1}$ because $'\omega(\bold v_1),\dots,'\omega(\bold v_d)$ coincide with the first $d$ basis vectors $\bold v_0,\dots,\bold v_{d-1}$ and so have dot product $0$ with the last row of $\mathcal{A}_{n+d}$.

So we have shown that there are at least $d$ linearly independent vectors in $\text{ker}(\mathcal{A}_{n+d+1})$ and these are obtained by deleting the first basis vector of $\mathcal{A}_{n+d}$ and appending a $0$ to each of the remaining basis vectors of $\mathcal{A}_{n+d}$.

Note that for $\mathcal{A}_{n+d+1}$ to have nullity $d$, neither the vector $[k_0,\dots,k_n,0,\dots,0]^t$ nor $[0,\dots,0,k_0,\dots,k_n]^t$ (in $F^{n+d+2}$) is in the kernel of $\mathcal{A}_{n+d+1}$ and there are $(q-1)^2$ choices of the pair $\{b_{n+d+1},a_{n+d+1}\}$ which make this so.
\end{proof}

\begin{corollary}
Suppose Toeplitz matrices $\mathcal{A}_{n-1}\rightharpoonup \cdots \rightharpoonup \mathcal{A}_{n+d}$ are given with nullities $0,1\dots,d+1$, respectively. Then there are $2q - 2$ choices of the pair of
entries $\{b_{n+d+1},a_{n+d+1}\}$ for which the corresponding matrix $\mathcal{A}_{n+d+1}$ has nullity $\nu_{n+d+1} = d+1$.
\end{corollary}

\begin{proof}
We know that the ranks of the matrices $\mathcal{A}_n \rightharpoonup \mathcal{A}_{n+1}$ are non-decreasing so that $\rho_{n+d} \leq \rho_{n+d+1} \leq \rho_{n+d}+2$ so $\nu_{n+d+1}$ is $d+1,d-1$ or $d$.  We have seen that there is only one choice of pair $\{b_{n+d+1},a_{n+d+1}\}$ for which $\mathcal{A}_{n+d+1}$ has nullity $d+2$ and $(q-1)^2$ for which $\mathcal{A}_{n+d+1}$ has nullity $d$.  Hence there must be $2q-2$ pairs for which $\mathcal{A}_{n+d+1}$ has nullity $d+1$.
\end{proof}

We shall look more carefully below at $\text{ker}(\mathcal{A}_{n+d+1})$ when $\nu_{n+d+1} = \nu_{n+d}$, see Lemma \ref{L: null_same}.
\begin{theorem}\label{T: null_down}
Suppose Toeplitz matrices $\mathcal{A}_{n-1}$ through $\mathcal{A}_{n+d}$ have nullities $0$ through $d+1$ and $\nu_{n+d+1} = \text{null}(\mathcal{A}_{n+d+1}) = d$.  Then for any embedded Toeplitz matrices $\mathcal{A}_{n+d+2}$ through $\mathcal{A}_{n+2d+1}$ the nullities are $d-1$ through $0$.
\end{theorem}
\begin{proof}
Since $\nu_{n+d} = d+1$ and $\nu_{n+d+1}=d$ we conclude from the proof of Theorem \ref{T: null_up} that the first and last entry of every vector in $\text{ker}(\mathcal{A}_{n+d+1})$ is $0$.  First suppose $\nu_{n+d+2}$ is not $d-1$ then by Proposition \ref{P: null}, $\nu_{n+d+2}$ is either $d+1$ or $d$, i.e. $\rho_{n+d+2}$ is either $\rho_{n+d+1}$ or
$\rho_{n+d+1}+1$.

Either case leads to the conclusion that $\omega(\text{ker}(\mathcal{A}_{n+d+1})) \subset \text{ker}(\mathcal{A}_{n+d+2})$ or $\sigma(\text{ker}(\mathcal{A}_{n+d+1})) \subset \text{ker}(\mathcal{A}_{n+d+2})$, or both.  Suppose $\omega(\text{ker}(\mathcal{A}_{n+d+1})) \subset \text{ker}(\mathcal{A}_{n+d+2})$.  Select a vector $\bold v \in \text{ker}(\mathcal{A}_{n+d+1})$ whose first nonzero entry appears leftmost among all vectors in $\text{ker}(\mathcal{A}_{n+d+1})$.  Then $\omega(\bold v)\in \text{ker}(\mathcal{A}_{n+d+2})$.  Therefore, viewing $\mathcal{A}_{n+d+1}$ in the lower right corner of $\mathcal{A}_{n+d+2}$, we conclude that $'\omega(\bold v)$ is in the kernel of $\mathcal{A}_{n+d+1}$, contradicting the leftmost position of the nonzero entry of $\bold v$.  Then by contradiction, $\omega(\text{ker}(\mathcal{A}_{n+d+1}))$ is not a subspace of $\text{ker}(\mathcal{A}_{n+d+2})$ and by a similar proof, neither is $\sigma(\text{ker}(\mathcal{A}_{n+d+1}))$.  Hence $\nu_{n+d+2}=d-1$.  Continuing this process inductively yields the result.
\end{proof}

The proof of the theorem leads immediately to the following conclusion.
\begin{corollary}
Assume the same hypotheses as in the previous theorem.  Then for each Toeplitz matrix $\mathcal{A}_{n+d+j},$ for $j=1,2,\dots,d$, and for any pair  $\{b_{n+d+j+1},a_{n+d+j+1}\}$ the corresponding matrix $\mathcal{A}_{n+d+j+1}\leftharpoonup \mathcal{A}_{n+d+j}$ has nullity one less than $\mathcal{A}_{n+d+j}$.
\end{corollary}

\begin{lemma}\label{L: null_same}
Let $d \geq 1$ and suppose for some $m\in \mathbb{N}$ that $\mathcal{A}_m\rightharpoonup \mathcal{A}_{m+1}$ satisfy $\nu_m = d = \nu_{m+1}$.  Then
 $\text{ker}(\mathcal{A}_{m+1}) = \omega(\text{ker}(\mathcal{A}_m))$ or $\text{ker}(\mathcal{A}_{m+1}) = \sigma(\text{ker}(\mathcal{A}_m))$.
 \end{lemma}
 \begin{proof}
 Since $\nu_m = \nu_{m+1}$ it follows that $\rho_{m+1} = \rho_m +1$ so one and only one of the following occurs:  either the last row of $A_{m+1}$ is in the span of the preceding rows, or the last column is in the span of the preceding columns.  In the former case it is clear that $\omega(\text{ker}(\mathcal{A}_m))\subset \text{ker}(\mathcal{A}_{m+1})$.  In the latter case, it follows from the symmetry of $\mathcal{A}_{m+1}$ that the first row is in the span of the succeeding rows and therefore $\sigma(\text{ker}(\mathcal{A}_m)) \subset \text{ker}(\mathcal{A}_{m+1})$.  A dimension argument finishes the proof.
  \end{proof}
\begin{theorem}\label{T: null_same}
Suppose for some $m\in \mathbb{N}$ that $\mathcal{A}_m\rightharpoonup \mathcal{A}_{m+1}\rightharpoonup \mathcal{A}_{m+2}$ satisfy $0 < \nu_m = \nu_{m+1}=\nu_{m+2}$.  Then either $\text{ker}(\mathcal{A}_{m+1}) = \omega(\text{ker}(\mathcal{A}_m))$ and $\text{ker}(\mathcal{A}_{m+2})=\omega^2(\text{ker}(\mathcal{A}_m))$ or
 $\text{ker}(\mathcal{A}_{m+1}) = \sigma(\text{ker}(\mathcal{A}_m))$ and $\text{ker}(\mathcal{A}_{m+2}) = \sigma^2(\mathcal{A}_m)$.
 \end{theorem}
 \begin{proof}
 Suppose $\text{ker}(\mathcal{A}_{m+1}) = \omega(\text{ker}(\mathcal{A}_m))$ but $\text{ker}(\mathcal{A}_{m+2}) = \sigma(\text{ker}(\mathcal{A}_{m+1})) = \sigma(\omega(\mathcal{A}_m)))$.  Then $\text{ker}(\mathcal{A}_{m+2})= \omega(\sigma(\text{ker}(\mathcal{A}_m)))$ and from this we have $\sigma(\text{ker}(\mathcal{A}_m)) \subset \text{ker}(\mathcal{A}_{m+1})$, a contradiction.  Similarly the assumption $\text{ker}(\mathcal{A}_{m+1}) = \sigma(\text{ker}(\mathcal{A}_m))$ and $\text{ker}(\mathcal{A}_{m+2})=\omega(\text{ker}(\mathcal{A}_{m+1}))$ leads to a contradiction, so the conclusion of the theorem follows.
\end{proof}
A proof similar to the above leads to the following result.

\begin{theorem}\label{T: null_same_same}
Suppose for some $m\in \mathbb{N}$ that $\mathcal{A}_m\rightharpoonup \mathcal{A}_{m+1}\rightharpoonup \mathcal{A}_{m+2}\rightharpoonup \cdots \rightharpoonup \mathcal{A}_{m+r}$ satisfy $0 < \nu_m = \nu_{m+1}= \cdots = \nu_{m+r}$.  If $\text{ker}(\mathcal{A}_{m+1}) = \omega(\text{ker}(\mathcal{A}_m))$ then $\text{ker}(\mathcal{A}_{m+r}) = \omega^r(\text{ker}(\mathcal{A}_m))$, and
if $\text{ker}(\mathcal{A}_{m+1}) = \sigma(\text{ker}(\mathcal{A}_m))$ then $\text{ker}(\mathcal{A}_{m+r}) = \sigma^r(\mathcal{A}_m)$.
 \end{theorem}

\begin{theorem}\label{T: equal equal down}
If $\nu_m = \nu_{m+1} = a > 0$ then $\nu_{m+2}$ is either $a$ or $a-1$.  In the latter case the subsequent nullities are $a-2,\dots,1,0$.
\end{theorem}
\begin{proof}
By Proposition \ref{P: null} the only other possibility is that $\nu_{m+2}=a+1$, so suppose this is the case.  We may assume this is the first time in the nullity sequence that this phenomenon has occurred, i.e. that there is an $a > 0$ and a triple $\nu_m, \nu_{m+1},\nu_{m+2}$ with nullities $a,a,a+1$.  From the preceding results we may assume that $\nu_m, \nu_{m+1},\nu_{m+2}$ is part of a nullity string $\nu_{j-1},\nu_j,\dots,\nu_{m+2}$ of the form $0,1,2,\dots,a-1,a,a,\dots,a,a+1$ or, if $j-1 = 0$, the string might be of the form $1,2,\dots,a-1,a,a,\dots,a,a+1$.  Since the argument for the latter is similar to the former we treat only the former case.  Then by Theorem \ref{T: null_one_vec} there is a single nonzero vector $\bold{k}\in F^{j+1}$ such that $\bold{k}$ generates $\text{ker}(\mathcal{A}_j)$ and $\bold{k}$ has nonzero initial $k_0$ and final entry $k_j$. If $\ell = j+a-1$, then by Theorem \ref{T: ker_form} $\text{ker}(\mathcal{A}_{\ell})$ is generated by the following vectors in $F^{j+a}$: $[\bold{k},0,\dots,0],[0,\bold{k},0,\dots,0],\dots,[0,\dots,0,\bold{k}]$. Then by the previous theorem  (with $r = m-\ell$) either $\text{ker}(\mathcal{A}_m) = \omega^r(\text{ker}(\mathcal{A}_{\ell}))$ and $\text{ker}(\mathcal{A}_{m+1})=\omega^{r+1}(\text{ker}(\mathcal{A}_{\ell}))$ or $\text{ker}(\mathcal{A}_m) = \sigma^r(\text{ker}(\mathcal{A}_{\ell}))$ and $\text{ker}(\mathcal{A}_{m+1})=\sigma^{r+1}(\text{ker}(\mathcal{A}_{\ell}))$.  We assume the former case, the latter case being similar.  From the form of $\bold{k}$ it follows that there is a vector $\bold{v}\in \text{ker}(\mathcal{A}_{m+1})$ beginning with $1$ and ending in $0$.  Since $\nu_{m+2} = a+1$ it follows that $\rho_{m+2} = \rho_{m+1}$. Thus the last row of $\mathcal{A}_{m+2}$ is in the span of the preceding rows, so $\omega(\bold{v})\in \text{ker}(\mathcal{A}_{m+2})$.  Since the last column of $\mathcal{A}_{m+2}$ is in the span of the preceding columns there is a vector $\bold{w}\in \text{ker}(\mathcal{A}_{m+2})$ which ends in $1$.  Then either $\bold{w}$ or, for some $\alpha\in F$, $\omega(\bold{v})+\alpha\bold{w}$ is a vector that begins with a $0$ and ends in $1$.  By deleting the initial $0$ from this vector and viewing $\mathcal{A}_{m+1}$ as being in the lower right corner of $\mathcal{A}_{m+2}$, we obtain a vector in $\text{ker}(\mathcal{A}_{m+1})$ that ends in a $1$, contradicting our assumption that $\text{ker}(A_{m+1}) = \omega(\text{ker}(A_m))$.  The argument for the case $\text{ker}(\mathcal{A}_{m+1}) = \sigma(\text{ker}(\mathcal{A}_m))$ is similar.

So we have shown that the triple $\nu_m, \nu_{m+1},\nu_{m+2}$ is either $a,a,a$ or $a,a,a-1$  Assuming the latter (and that $a-1 >0$) we will show that $\nu_{m+3} = a-2$.  As above, $\text{ker}(\mathcal{A}_{m+1})$ is either $\omega^{r+1}(\text{ker}(\mathcal{A}_{\ell}))$ or $\sigma^{r+1}(\text{ker}(\mathcal{A}_{\ell}))$.  We assume the former case (with the proof of the latter case being similar).  Then $\text{ker}(\mathcal{A}_{m+1})$ is generated by the vectors $\omega^{r+1}([\bold{k},0,\dots,0]),\omega^{r+1}([0,\bold{k},0,\dots,0]),\dots,\omega^{r+1}([0,\dots,0,\bold{k}])$.  It can be verified that the image under $\omega$ of the second through last of these vectors is in ker$(\mathcal{A}_{m+2})$ and therefore $\omega^{r+2}([\bold{k},0,\dots,0])$ is not in
$\text{ker}(\mathcal{A}_{m+2})$.  It follows that every vector in $\text{ker}(\mathcal{A}_{m+2})$ begins and ends with a $0$.  Suppose then that $\nu_{m+3} \geq a-1$.  Then by an
argument similar to the proof of Lemma \ref{L: null_same} it follows that either $\omega(\text{ker}(\mathcal{A}_{m+2}))\subset \text{ker}(\mathcal{A}_{m+3})$, or
$\sigma(\text{ker}(\mathcal{A}_{m+2}))\subset \text{ker}(\mathcal{A}_{m+3})$, and then by an argument similar to that in the last paragraph of Theorem \ref{T: null_down} it follows that
$\nu_{m+3} = a-2$.  Continuing in this way we verify that the nullities from $\nu_{m+1}$ on are $a,a-1,a-2,\dots,1,0$.
\end{proof}

\begin{theorem}\label{T: null_same_count}
Suppose for some $m\in \mathbb{N}$ that $\mathcal{A}_m\rightharpoonup \mathcal{A}_{m+1}$ and  $0 < a = \nu_m = \nu_{m+1}$.  Then $\nu_{m+2}$ is either $a$ or $a-1$.
There are $q$ matrices $\mathcal{A}_{m+2}$ for which $\mathcal{A}_{m+1}\rightharpoonup \mathcal{A}_{m+2}$ and $\nu_{m+2} = a$, and $q^2-q$ for which $\nu_{m+2} = a-1$.
\end{theorem}

\begin{proof}
From Theorem \ref{T: null_down} and Theorem \ref{T: equal equal down} it follows that there is an $n\in \mathbb{N}$ such that, starting with $n-1$ and ending with $m+1$, the nullity string $\nu_{n-1}$ through $\nu_{m+1}$ is $0,1,\dots,a-1,a,\dots,a$  (or if $n=1$ it is possible that the nullity string is $1,\dots,a-1,a,\dots,a$).  Let $\bold k$ span $\text{ker}(\mathcal{A}_n)$ then from Theorem \ref{T: null_same_same} there is a positive integer $s$ such that $\text{ker}(\mathcal{A}_{m+1}) = \omega^s(\text{ker}(\mathcal{A}_{m+1-s}))$ or $\sigma^s(\text{ker}(\mathcal{A}_{m+1-s}))$ where $\text{ker}(\mathcal{A}_{m+1-s}) = \text{span}\{[\bold k,0,\dots,0],[0,\bold k,0,\dots,0],\dots[0,\dots,0,\bold k]\}$. If $\text{ker}(\mathcal{A}_{m+1}) = \omega^s(\text{ker}(\mathcal{A}_{m+1-s}))$ then since the first entry of $\bold k$ is nonzero, we may choose $b_{m+2}$ so that $\bold v = [\bold k,0,\dots,0]\in F^{m+3}$ either is or is not in $\text{ker}(\mathcal{A}_{m+2})$. If $\bold v \in \text{ker}(\mathcal{A}_{m+2})$ then the first column of $\mathcal{A}_{m+2}$ is in the span of the succeeding columns, hence the last row of $\mathcal{A}_{m+2}$ is in the span of the preceding rows, and it then follows that $\omega(\text{ker}(\mathcal{A}_{m+1}))\subset \text{ker}(\mathcal{A}_{m+2})$. Then by the preceding theorem $\mathcal{A}_{m+1}$ and $\mathcal{A}_{m+2}$ have the same nullity, so
$\omega(\text{ker}(\mathcal{A}_{m+1}))= \text{ker}(\mathcal{A}_{m+2})$.  Note that this holds regardless of the choice of $a_{m+2}$.  Thus there are $q$ choices of pairs ${b_{m+2},a_{m+2}}$ such that $\nu_{m+2}=a$ (where $b_{m+2}$ is fixed and $a_{m+2}$ is arbitrary).  On the other hand, if $\bold v\notin \text{ker}(\mathcal{A}_{m+2})$ then since $\text{ker}(\mathcal{A}_{m+2})$ is neither $\omega(\text{ker}(\mathcal{A}_{m+1}))$ nor $\sigma(\text{ker}(\mathcal{A}_{m+1}))$ it follows from the preceding corollary that $\nu_{m+2}= a-1$.  Since this is true regardless of the choice of $a_{m+2}$ there are $q^2-q$ choices of a matrix $\mathcal{A}_{m+2}$ satisfying these conditions (there are $q-1$ choices for $b_{m+2}$ and $q$ choices for $a_{m+2}$).

A similar argument holds for the case where $\text{ker}(\mathcal{A}_{m+1}) = \sigma^s(\text{ker}(\mathcal{A}_{m+1-s}))$.
\end{proof}

\begin{remark}\label{R: null_string}
Note that Theorem \ref{T: null_string} follows by assembling Proposition \ref{P: null}, the corollary to Theorem \ref{T: ker_form}, Theorem \ref{T: null_up} and its corollary, Theorem \ref{T: null_down} and Theorem \ref{T: null_same_count}.
\end{remark}

\section{Additional Counting Results for Nullity Patterns}\label{S: counting}

In this section we complete the counting arguments to determine the number of strings of Toeplitz matrices satisfying the patterns which appear in Theorem \ref{T: null_string}.  See Remark \ref{R: counting_summary} for a summary of these computations.

\begin{theorem}\label{T: 100_101}
If for some $n\in \mathbb{Z}^+$ we have $\nu_n = 1$ and $\nu_{n+1} = 0$, then there are $q^2-q$ pairs of entries $\{b_{n+2},a_{n+2}\}$ for which the Toeplitz matrix $\mathcal{A}_{n+2}$ satisfies $\nu_{n+2} = 0$ and $q$ for which $\nu_{n+2} = 1$.
\end{theorem}
\begin{proof}
First note that since $\nu_n = 1$ and $\nu_{n+1} = 0$, it follows that $\rho_{n+1} = \rho_n +2$, and therefore $[a_{n+1},\dots,a_1]^t$ is not the zero vector.
For if it were, then matrix $\mathcal{A}_n$ and the $(n+1)\times (n+2)$ matrix obtained by appending $[a_{n+1},\dots,a_1]^t = [0,\dots,0]^t$ as the last column, both have the same rank.  So then $\rho(\mathcal{A}_{n+1}) \leq \rho(\mathcal{A}_n)+1$, a contradiction.
Fix $a_{n+2}$.  Since $\mathcal{A}_{n+1}$ is invertible there is a vector $\bold {\ell} = [\ell_0,\ell_1,\dots,\ell_{n+1}]^t$ satisfying $\mathcal{A}_{n+1}\cdot \ell = [-a_{n+2},-a_{n+1},\dots,-a_1]^t$.  Note that $\ell_0 \neq 0$, for otherwise it would follow (since $\mathcal{A}_{n}$ is embedded in the lower right corner of $\mathcal{A}_{n+1}$) that $[a_{n+1},a_n,\dots,a_1]^t$ is in the span of the column space of $\mathcal{A}_n$, and therefore the rank of $\mathcal{A}_{n+1}$ is at most one greater than the rank of $\mathcal{A}_n$.  But the assumption that $\nu_{n+1} = 0$ and $\nu_n =1$ means that $\rho_{n+1} = 2 + \rho_n$.
Setting $\bold L = [\bold \ell,1]^t = [\ell_0, \ell_1,\dots,\ell_{n+1},1]^t$, then $\mathcal{A}_{n+2}\cdot \bold L = [0,\dots,0,*]^t$.  Then there are $q-1$ choices of $b_{n+2}$ so that the last entry of $\mathcal{A}_{n+2}\cdot \bold L$ is nonzero, so $\bold L$ is not in $\text{ker}(\mathcal{A}_{n+2})$.  Suppose there is a non-zero vector $\bold P$ in $\text{ker}(\mathcal{A}_{n+2})$: then since the first $n+2$ columns of $\mathcal{A}_{n+2}$ are linearly independent (as a consequence of $\mathcal{A}_{n+1}$ lying in the top left corner of $\mathcal{A}_{n+2}$), it follows that the last entry of $\bold P$ is nonzero.  Therefore $\mathcal{A}_{n+2}\cdot (\bold L + \bold P) = [0,\dots,0,c]^t$ with $c\neq 0$.   Replacing $\bold P$ with a nonzero scalar multiple of $\bold P$, if necessary, we may assume that $\bold L + \bold P$ has last entry $0$. Since $\bold L + \bold P$ has last entry $0$ then by viewing $\mathcal{A}_{n+1}$ in the top left corner of $\mathcal{A}_{n+2}$, the last equation implies that the columns of $\mathcal{A}_{n+1}$ are dependent, a contradiction.  Therefore, by contradiction, $\mathcal{A}_{n+2}$ is invertible.

Thus for each $a_{n+2}$ there is are $q-1$ entries $b_{n+2}$ such that the corresponding matrix $\mathcal{A}_{n+2}$ is invertible.  So there are $q^2-q$ pairs ${b_{n+2},a_{n+2}}$ for which $\mathcal{A}_{n+2}$ is invertible.  This leaves $q$ choices of the pair $\{b_{n+2},a_{n+2}\}$ corresponding to matrices $\mathcal{A}_{n+2}$ of nullity $1$.
\end{proof}

\begin{theorem}\label{T: zero_zero}
Suppose $\mathcal{A}_{n} \rightharpoonup \mathcal{A}_{n+1}$ are Toeplitz matrices for some $n\in \mathbb{Z}^+$ with $\nu_{n} = 0 = \nu_{n+1}$.  Then there are $q^2-1$ choices of entries $(b_{n+2},a_{n+2})$ for a matrix $\mathcal{A}_{n+2}\leftharpoonup \mathcal{A}_{n+1}$,  such that $\nu_{n+2} = 0$, and one such that $\nu_{n+2} = 1$.
\end{theorem}
\begin{proof}
By Proposition \ref{P: null} either $\nu_{n+2} = 0$ or $\nu_{n+2} = 1$.  Suppose there are two Toeplitz matrices $\mathcal{A}_{n+2}$ and $\mathcal{A}_{n+2}'$, both of nullity $1$, such that $\mathcal{A}_n \rightharpoonup \mathcal{A}_{n+1} \rightharpoonup \mathcal{A}_{n+2}$ and
$\mathcal{A}_n \rightharpoonup \mathcal{A}_{n+1} \rightharpoonup \mathcal{A}_{n+2}'$.  Then there are nonzero vectors $\ell$ and $\bold L$ in $F^{n+3}$, such that $\ell \in \text{ker}(\mathcal{A}_{n+2})$ and $\bold L\in \text{ker}(\mathcal{A}_{n+2}')$.  By Theorem \ref{T: null_one_vec} both $\ell$ and $\bold L$ have nonzero initial and final entries.  Therefore we may assume that the initial entry of $\ell +\bold L$ is $0$.  Then, viewing $\mathcal{A}_{n+1}$ as occupying the bottom right corner of $\mathcal{A}_{n+2}$ we conclude that $'(\ell +\bold L)$ is in $\text{ker}(\mathcal{A}_{n+1})$.  This contradicts $\nu_{n+1}=0$ unless $\ell+\bold L = \bold 0$.  So $\bold L = -\ell$, and so $a_{n+2}=a_{n+2}'$.  Similarly, $b_{n+2} = b_{n+2}'$.  So there is at most one $\mathcal{A}_{n+2}$ with nullity $1$.

To show that there is at least one Toeplitz math $\mathcal{A}_{n+2}$ with nullity $1$ and $\mathcal{A}_{n+1} \rightharpoonup \mathcal{A}_{n+2}$ choose  $a_{n+2} \neq a_{n+2}'$.  Then, since $\mathcal{A}_{n+1}$ is invertible there exist vectors $\bold v$ and $\bold V$ in $F^{n+2}$ such that
\begin{alignat}{2}
\mathcal{A}_{n+1}\bold v &= [-a_{n+2},-a_{n+1},\dots,-a_1]^t, \text{and}\\
\mathcal{A}_{n+1}\bold V &= [a_{n+2}',a_{n+1},\dots,a_1]^t.
\end{alignat}
The sum of the initial entries, $v_0 + V_0$, is nonzero, otherwise $'\bold v + '\bold V$ is in $\text{ker}(\mathcal{A}_n)$, a contradiction unless $\bold v = -\bold V$, which is false.  Suppose $v_0 \neq 0$ then we can choose $b_{n+2}$ such that $\mathcal{A}_{n+2}[v_0,v_1,\dots,v_{n+1},1]^t = \mathbf{0}$, so that $\mathcal{A}_{n+2}$ has nontrivial kernel.  Thus, by the comment made above, $\mathcal{A}_{n+2}$ has nullity $1$.
\end{proof}

\begin{remark}\label{R: counting_summary}
In this remark we summarize the results of the counting arguments we have made thus far. The first two integers of each line indicate the nullities $\nu_{n-1}$ and $\nu_n$ of a pair of fixed Toeplitz matrices $\mathcal{A}_{n-1}$ and $\mathcal{A}_n$, with $\mathcal{A}_{n-1} \rightharpoonup \mathcal{A}_{n}$.  The integer over the arrow represents the number of Toeplitz matrices $\mathcal{A}_{n+1}$ satisfying
$\mathcal{A}_{n-1} \rightharpoonup \mathcal{A}_n \rightharpoonup \mathcal{A}_{n+1}$ and having nullity $\nu_{n+1}$, the third integer entry of the line.
\begin{enumerate}
\item $0,0 \overset{q^2-q+1}{\rightarrow} 0$, Theorem \ref{T: zero_zero}
\item $0,0 \overset{q-1}{\rightarrow} 1$, Theorem \ref{T: zero_zero}
\item $1,0 \overset{q^2-q}{\rightarrow} 0$, Theorem \ref{T: 100_101}
\item $1,0 \overset{q}{\rightarrow} 1$, Theorem \ref{T: 100_101}
\item For $d \geq 1$, $d-1,d \overset{1}{\rightarrow} d+1$, Corollary to Theorem \ref{T: ker_form}
\item For $d \geq 1$, $d-1,d \overset{2q-2}{\rightarrow} d$, Corollary to Theorem \ref{T: null_up}
\item For $d \geq 1$, $d-1,d \overset{(q-1)^2}{\rightarrow} d-1$, Theorem \ref{T: null_up}
\item For $d \geq 1$, $d,d \overset{q}{\rightarrow} d$, Theorem \ref{T: null_same_count}
\item For $d \geq 1$, $d,d \overset{q^2-q}{\rightarrow} d-1$, Theorem \ref{T: null_same_count}
\item For $d \geq 2$, $d,d-1 \overset{q^2}{\rightarrow} d-2$, Theorem \ref{T: null_down}
\end{enumerate}
\end{remark}

For the remainder of this section we will assume that $\mathcal{F}$ is $GF(2)$.

\begin{definition}\label{D: C_string}
For $\ell\in \mathbb{Z}^+$ let $\mathcal{A}_{\ell}$ be a fixed Toeplitz matrix with nullity $\nu_{\ell}$ and let $\nu_{\ell},\nu_{\ell +1},\dots,\nu_p$.  Then the number of strings $\mathcal{A}_{\ell},\mathcal{A}_{\ell+1},\dots,\mathcal{A}_p$ with the prescribed nullities will be denoted by $C(\nu_{\ell},\nu_{\ell+1},\dots,\nu_p)$.
\end{definition}

\begin{lemma}\label{L: up_down_count}
Let $m$ be a positive integer, then $C(1,2,\dots,m-1,m,m-1,\dots,2,1,0) = 2^{2m-2}$.
\end{lemma}

\begin{proof}
For $0 < \nu_s$ and $\nu_{s+1}= \nu_s +1$, by the corollary to Theorem \ref{T: ker_form} there is, for a given matrix $\mathcal{A}_s$ with nullity $\nu_s$, only one $\mathcal{A}_{s+1}$ with $\mathcal{A}_s\rightharpoonup \mathcal{A}_{s+1}$ and with nullity $\nu_{s+1}$.  If $\mathcal{A}_s = m$ only one $\mathcal{A}_{s+1}$ with $\mathcal{A}_s\rightharpoonup \mathcal{A}_{s+1}$ satisfies $\nu_{s+1} = m-1$, by Theorem \ref{T: null_up}. By Theorem \ref{T: null_down} if $m-1 \geq \nu_s \geq 1$ then there are four choices of $\mathcal{A}_{s+1}$ with nullity $\nu_{s+1} = \nu_{s} - 1$.
\end{proof}

\begin{lemma}\label{L: plateau_count}  If $\nu_{n-1}=0$ and $\nu_n=\nu_{n+1}=\dots=\nu_{n+u-1}=1$, where $u>1$, then $C(\nu_n,\dots,\nu_{n+u-1})=C(1,1,\dots,1) = 2^{u-1}$.  If $m>1$ and $m$ is repeated $u$ times, then  $C(1,2,\dots,m-1,m,m,\dots,m,m-1,\dots,2,1,0)=2^u 2^{2m-2}$
\end{lemma}
\begin{proof}  The first statement follows from Theorem \ref{T: null_same_count}.  For the second statement, as the nullity increases from $1$ to $m$ there is only one choice of matrix at each stage, by Theorem \ref{T: ker_form}.  For each of the $u-1$ occurrences of nullity $m$ after the first, there are two choices of matrix which fix the nullity at $m$, and then following the last occurrence of $m$ there are two choices of matrix that drop the nullity from $m$ to $m-1$.  Thereafter, as the nullity drops from $m-1$ to $0$ there are four choices of matrix at each stage.  So we have a total of $2^{u-1}\cdot 2\cdot 4^{m-1}$ choices
\end{proof}
The following result will be an essential part of the argument that half of the Toeplitz matrices of a given size are invertible.  It will also be useful further on when we calculate the number of Toeplitz matrices of a given size of a specified rank.
\begin{theorem}\label{positive_string}
Let $\ell\in \mathbb{Z}^+$ and let $\mathcal{A}_{\ell}$ be a Toeplitz matrix of nullity $1$, where we assume that $\nu_{\ell -1} = 0$ if $\ell \geq 1$.
Let $n$ be a positive integer.  Then the number of matrices $\mathcal{A}_{\ell + n}$ satisfying

\begin{enumerate}
\item $\mathcal{A}_{\ell} \rightharpoonup \mathcal{A}_{\ell + 1} \rightharpoonup \cdots \rightharpoonup \mathcal{A}_{\ell + n}$,
\item $\nu_s > 0$ for $s=1,2,\dots n-1$, and
\item $\nu_{\ell + n} = 0$
\end{enumerate}
\noindent is $n2^{n-1}$.
\end{theorem}
\begin{proof} For $n = 1$ the only corresponding nullity string is $\nu_{\ell},\nu_{\ell +1}$ which can only be $1,0$.  By Theorem \ref{T: null_up} there is only one choice for $\mathcal{A}_{\ell +1}$ having nullity $0$.  For $n=2$ the only nullity string is $1,1,0$ and there are $4$ matrices $\mathcal{A}_{\ell +2}$ with this string, Suppose the statement is true for some even integer $n\geq 2$.  Then for strings of matrices of length $n+2$ satisfying the hypotheses, there are two types of nullity strings.  First there are the nullity strings $1,2,\dots,m-1,m,\dots,m,m-1,\dots,2,1$ obtained from the nullity strings of length $n$ by inserting $m$ twice, and there is the nullity string $1,2,\dots,\frac{n}{2},\frac{n+2}{2},\frac{n+2}{2},\frac{n}{2},\dots,2,1,0$.  By Lemma \ref{L: plateau_count} there are $4\cdot 4^{\frac{n}{2}} = 4\cdot 2^{n}$ of the latter and by the induction assumption and Lemma \ref{L: plateau_count} there are $4n2^{n-1}$ of the former.  Adding gives $(n+2)2^{n+1}$, so the even case is established by induction.  The proof of the odd case is similar.
\end{proof}

\section{Counting Invertible Toeplitz Matrices}\label{S: number_invertible}

As an application of the counting arguments for nullity strings we will show that exactly half of the Toeplitz matrices $\mathcal{A}_n$ are invertible, see the corollary to Theorem \ref{T: invertible_formula}. As a further application we will count the number of $\mathcal{A}_n$'s of any specified rank, see the corollary to Theorem \ref{T: T_count}.

\begin{definition}\label{D: theta_eta}
For $n\in \mathbb{N}$ let $\theta(n)$ be the number of $\mathcal{A}_n$'s such that $\nu_{n-1} = 0 = \nu_n$.  Define $\eta(n)$ to be the number of $\mathcal{A}_n$'s such that $\nu_{n-1} = 1$ and $\nu_n = 0$.
\end{definition}

Note that if $n\in \mathbb{N}$ and $\mathcal{A}_n$ is invertible then $\mathcal{A}_n$ falls either among those invertible matrices counted by $\theta(n)$ or by $\eta(n)$, and therefore we have the following.

\begin{proposition}\label{P: num_invertible}
For $n\in \mathbb{N}$, the number of invertible $\mathcal{A}_n$'s is $\theta(n) + \eta(n)$.
\end{proposition}

\begin{remark}\label{R: small_theta_eta}  The formula in the theorem below computes the number of invertible matrices $\mathcal{A}_n$ for $n\geq 2$.  For $n=0$ note that $\mathcal{A}_0 =[1]$ is the only invertible $1\times 1$ matrix.  For $n=1$ there are only two possible nullity strings $\nu_0,\nu_1$ ending in $0$, namely $0,0$ and $1,0$.  By Proposition \ref{P: null} and the corollary to Theorem \ref{T: ker_form} the there are three matrices with the former and one with the latter nullity string.  So $\theta(1) = 3$ and $\eta(1) = 1$, and there are $\theta(1) + \eta(1) = 4$ invertible $2\times 2$ Toeplitz matrices.
\end{remark}

\begin{proposition}\label{P: theta_recursion}
For $n\in \mathbb{N}$, $\theta(n+1) = 3\theta(n) + 2\eta(n)$.
\end{proposition}
\begin{proof}
It follows from Theorem \ref{T: zero_zero} that if $\mathcal{A}_{n-1}$ and $\mathcal{A}_n$ both have nullity $0$ then there are three choices of pairs $\{b_{n+1},a_{n+1)}\}$ for $\mathcal{A}_{n+1}$ to have nullity $0$.  If $\nu_{n-1}=1$ and $\nu_n = 0$ there are two choices for $\{b_{n+1},a_{n+1)}\}$ that make $\mathcal{A}_{n+1}$ invertible, by Theorem \ref{T: 100_101}.
\end{proof}

\begin{theorem}\label{T: invertible_formula}
For $n\geq 2$ the number of invertible $(n+1)\times(n+1)$ Toeplitz matrices $\mathcal{A}_n$ over  $GF(2)$  is
\begin{multline}
\notag n2^{n-1}+(n-1)2^{n-2}\\
+\sum_{j=1}^{n-2}\left[ \theta(j)+2\eta(j)\right]((n-1)-j)2^{n-2-j}\\
+(3\theta(n-1)+2\eta(n-1))
\end{multline}

\end{theorem}
\begin{proof}
Each invertible Toeplitz matrix $\mathcal{A}_n$ has a nullity string $\nu_0,...,\nu_{n}$ which satisfies $\nu_n=0$ and the rules of Theorem \ref{T: null_string}.
We divide the collection of all of these invertible matrices into disjoint sets $S, S_0,...,S_{n-1}$, depending on the structure of their nullity string.  A matrix $\mathcal{A}_n$ belongs to $S$ iff $\nu_k >0$ for all $k<n$, i.e., $\nu_n$ is the only $0$ on the string.  $\mathcal{A}_n$ is in $S_j$ if $\nu_j=0$ and $\nu_k > 0$ for $j<k<n-1$.

By Theorem \ref{positive_string}  there are $n2^{n-1}$ matrices $\mathcal{A}_n$ in $S$.  This is the first term of the formula in the statement of the theorem.  The set $S_0$ consists of all matrices whose nullity string is $0,\nu_1,...\nu_{n-1},\nu_n$, where $\nu_1,\nu_2,...,\nu_{n-1}$ has no zeroes and $\nu_n = 0$.  Then $\nu_1 = 1$, by Proposition \ref{P: null}, and there is only one choice for $\mathcal{A}_1$, namely
$\left[ \begin{smallmatrix} 1 & 1 \\ 1 & 1 \end{smallmatrix}\right]$,
 so by Theorem \ref{positive_string} there are $(n-1)2^{n-2}$ matrices in $S_0$.  This is the second term of the formula in the statement of the proposition.  If a matrix is in $S_1$ its nullity sequence is either $0,0,\nu_2,...\nu_n$ or $1,0,\nu_2,...\nu_n$, with $\nu_n =0$ and $\nu_2$ through $\nu_{n-1}$ non-zero.  There are $\theta(1)$ strings $\nu_0 = 0, \nu_1 = 0$ so by Remark 2.2 and Theorem \ref{positive_string} there are $\theta(1)(n-2)2^{n-3}$ strings in $S_1$ beginning with $\nu_0 = 0$ and $\nu_1 = 0$.  There are $\eta(1)$ strings $\nu_0 =1, \nu_1 = 0$, so by Remark 2.4 and Theorem \ref{positive_string}, there are $2\eta(1)(n-2)2^{n-3}$ strings in $S_1$ beginning with $\nu_0 = 1$ and  $\nu_1 = 0$.  This give the first term in the summation.
The other terms in the summation are obtained by a similar analysis of the properties of nullity sequences for the matrices in sets $S_2$ through $S_{n-2}$.

Now consider the set $S_{n-1}$, then the triple $\nu_{n-2},\nu_{n-1},\nu_n$ is either $0,0,0$ or $1,0,0$.  By Remarks 2.1 and 2.3, the number of matrices in $S_{n-1}$ is $3\theta(n-1)+2\eta(n-1)$.
\end{proof}
\begin{theorem}\label{T: count invertible}
The number of invertible Toeplitz matrices $\mathcal{A}_n$ is $2^{2n}$, half of the total number of $(n+1)\times (n+1)$ Toeplitz matrices.
\end{theorem}
\begin{proof}
The remark above shows that the conclusion holds for $n=0$ and $n=1$.  For $n\geq 2$ we will prove the result using the formula in the previous theorem, along with induction on $n\in \mathbb{N}$ on the following formulas.

\begin{align}
\theta(n) &= \frac{2^{2n+1}+1}{3}\tag{14.1} \\
\eta(n) &= \frac{2^{2n}-1}{3}\tag{14.2}
\end{align}

The two formulas are valid for $n = 1$ from the remark above.  For $n=2$ we see that $\theta(2)$ is the count of all matrices with initial nullity string $0,0,0$ or $1,0,0$.  Applying Theorem \ref{T: zero_zero} there are nine of the former and two of the latter.  For $\eta(2)$ the relevant nullity strings are $0,1,0$ and $1,1,0$ for $\nu_0,\nu_1,\nu_2$.  There is one of the former and there are four of the latter.

It is possible to show that when $n=2$ the formula in the theorem gives the answer $16$.  So now we assume that for some positive integer $k \geq 2$ the formulas (14.1) and (14.2) are valid, and that the formula in the theorem gives $2^{2k}$ for $n=k$, i.e.
\begin{multline}
k2^{k-1}+(k-1)2^{k-2}\\
+\sum_{j=1}^{k-2}\left[ \theta(j)+2\eta(j)\right]((k-1)-j)2^{k-2-j}\\
+(3\theta(k-1)+2\eta(k-1)) = 2^{2k} \tag{14.3}
\end{multline}

We need to show that (14.3) is valid with $k$ replaced by $k+1$, i.e. that
\begin{equation}
(k+1)2^{k}+k2^{k-1}+\sum_{j=1}^{k-1}\left[ \theta(j)+2\eta(j)\right]((k-j)2^{k-1-j})+(3\theta(k)+2\eta(k)) = 2^{2k+2} \tag{14.4}
\end{equation}
To show this it will suffice to show that if we subtract two times the left side of (14.3)  from the left side of (14.4) we obtain $2^{2k+1}$.  Taking this difference term by term gives
\begin{eqnarray*}
& &2^k+2^{k-1}+\sum_{j=1}^{k-2} \left(\left[ \theta(j)+\eta(j)\right]2^{k-1-j}\right) -5\theta(k-1)-2\eta(k-1) +3\theta(k)+2\eta(k)\\
&=& 2^k + 2^{k-1} +\sum_{j=1}^{k-2}\left( \frac{2^{2j+1}+1}{3} + 2\frac{2^{2j}-1}{3}\right)2^{k-1-j} \\
& & \quad \quad \quad \quad \quad \quad \quad -5\frac{2^{2k-1}+1}{3}-2\frac{2^{2k-2}-1}{3}+3\frac{2^{2k+1}+1}{3} + 2\frac{2^{2k}-1}{3}\\
&=& 2^k + 2^{k-1} +\sum_{j=1}^{k-2}\left( \frac{2^{2j+2}-1}{3}\right)2^{k-1-j} \\
& & \quad \quad \quad \quad \quad \quad \quad -5\frac{2^{2k-1}+1}{3}-2\frac{2^{2k-2}-1}{3}+3\frac{2^{2k+1}+1}{3} + 2\frac{2^{2k}-1}{3}\\
&=& 2^k + 2^{k-1} +\sum_{j=1}^{k-2}\left( \frac{2^{k+j+1}-2^{k-1-j}}{3}\right)\\
& & \quad \quad \quad \quad \quad \quad \quad -5\frac{2^{2k-1}+1}{3}-2\frac{2^{2k-2}-1}{3}+3\frac{2^{2k+1}+1}{3} + 2\frac{2^{2k}-1}{3}\\
&=& 2^k + 2^{k-1} +\left( \frac{2^{k+2}+2^{k+3}+\cdots + 2^{2k-1}}{3}- \frac{2^{k-2}+2^{k-3}+\cdots + 2}{3}\right)\\
& & \quad \quad \quad \quad \quad \quad \quad -5\frac{2^{2k-1}+1}{3}-2\frac{2^{2k-2}-1}{3}+3\frac{2^{2k+1}+1}{3} + 2\frac{2^{2k}-1}{3}\\
&=& 2^k + 2^{k-1}+ \left(\frac{2^{2k}-2^{k+2}}{3}\right)- \left(\frac{2^{k-1}-2}{3}\right) -5\frac{2^{2k-1}+1}{3}-2\frac{2^{2k-2}-1}{3}+ \\ & & \quad \quad \quad \quad \quad 3\frac{2^{2k+1}+1}{3} + 2\frac{2^{2k}-1}{3},
\end{eqnarray*}
and a straightforward computation shows that this expression is $2^{2k+1}$.

To finish the induction we need to verify the formulas for $\theta(k+1)$ and $\eta(k+1)$.  Since $\theta(k) + \eta(k) = 2^{2k}$ we have
\begin{eqnarray*}
\theta(k+1) &=& 3\theta(k) + 2\eta(k) \\
  &=& 2(\theta(k) + \eta(k)) + \theta(k) \\
  &=& 2^{2k+1} + \left( \frac{2^{2k+1}+1}{3}\right)\\
  &=& \frac{3\cdot2^{2k+1}+2^{2k+1}+1}{3}\\
  &=& \frac{2^{2k+3}+1}{3},
\end{eqnarray*}
and so
\[\eta(k+1) = 2^{2k+2} - \theta(k+1) = 2^{2k+2}-\left(\frac{2^{2k+3}+1}{3}\right)=\frac{2^{2k+2}-1}{3},\]
to finish the proof.

\end{proof}

\section*{Counting matrices of positive nullity}\label{S: Null_1}


In this section we show that the theorem above can be applied to determine the number $N(n,\nu)$ of Toeplitz matrices $\mathcal{A}_n$ with nullity $\nu$, for any $n\in \mathbb{Z}^+$. The first step is the following corollary to the theorem.

\begin{corollary}
Consider the set of all $(n+1)\times(n+1)$ Toeplitz matrices $\mathcal{A}_n$ whose nullity string $\nu_0,\nu_1,\dots,\nu_n$ is non-zero and satisfies $\nu_n =1$.  There is only one such matrix for $n=0$, and for $n\geq 1$ there are $(n+3)2^{n-2}$ such matrices.
\end{corollary}

\begin{proof}
For $n=0$, only $\mathcal{A}_0 = [0]$ has nullity $1$.
For $n=1$ the only nullity string that satisfies the conditions is $1,1$, and there are two matrices $\mathcal{A}_1$ with this string, namely, $\left[ \begin{smallmatrix} 0 & 1 \\ 0 & 0 \end{smallmatrix}\right]$ and $\left[ \begin{smallmatrix} 0 & 0 \\ 1 & 0 \end{smallmatrix}\right]$ .  So the formula holds for $n=1$.

Suppose $n\geq 2$.  By Theorem \ref{positive_string} there are $(n+1)2^n$ matrices $\mathcal{A}_{n+1}$ whose nullity string $\nu_0,\nu_1,\dots,\nu_{n+1}$ satisfies the hypotheses of that theorem.  Note that every string of this form is obtained by appending a $0$ to the strings $\nu_0,\nu_1,\dots,\nu_n$ satisfying the hypotheses of the corollary.  One of these is $1,1,\dots,1$ which, by the counting rules, is the string for $2^n$ Toeplitz matrices $\mathcal{A}_n$ so that again by the counting rules $1,1,\dots,1,0$ is the nullity string for $2^{n+1}$ matrices $\mathcal{A}_n$.  Now consider the other nullity strings satisfying the hypotheses of the corollary.  Let $R$ be the number of matrices $\mathcal{A}_n$ with this nullity string.  Each of these ends in $2,1$ so by the counting rules, if we append a $0$ to the end there are four times as many matrices $\mathcal{A}_{n+1}$ with the latter string as there are $\mathcal{A}_n$'s with the former.
So we have $(n+1)2^n = 2^{n+1} + 4R$, so $R = (n-1)2^{n-2}$ and the number of Toeplitz matrices satisfying the hypotheses of the corollary is $R + 2^n = (n+3)2^{n-2}.$

\end{proof}
\begin{theorem}\label{T: nonzero_string_end_s}
For non-negative integers $n$ and for $s\leq n+1$ the number of Toeplitz matrices $\mathcal{A}_n$ with a positive nullity string $\nu_0,\nu_1,\nu_2,\dots,\nu_n$ ending in $s$ agrees with the number of Toeplitz matrices $\mathcal{A}_{n+1}$ with a positive nullity string ending in $s+1$.
\end{theorem}
\begin{proof}
We exhibit a one to one correspondence between non-zero nullity strings $\nu_0,\dots,\nu_n$ of length $n+1$ ending in $s$ and nonzero nullity strings of length $n+2$ ending in $s+1$.  First suppose $\nu_0,\dots,\nu_n$ is $1,1,\dots,1$, then we match this with the string $1,2,\dots,2$ of length $n+2$.  Otherwise, suppose $\nu_0,\nu_1,\dots,\nu_n$ is some other non-zero nullity string and $\nu_n = s$.  Then there is a $j>1$ such that $j = max(\nu_0,\nu_1,\dots,\nu_n)$ and $\nu_0,\nu_1,\dots,\nu_j-1$ is $1,2,\dots,j$.  Consider the nullity string of length $n+2$ given by $1,2,\dots,j,j+1,\nu_j+1,\nu_{j+1}+1,\dots,\nu_n+1$.  It is not difficult that these matches exhibit a one to one correspondence between the non-zero nullity strings of length $n+1$ ending in $s$ and the non-zero nullity strings of length $n+2$ ending in $s+1$.  Moreover it can also be verified using the counting rules that there are the same number of Toeplitz matrices $\mathcal{A}_n$ corresponding to the string of length $n+1$ as there are matrices $\mathcal{A}_{n+1}$ corresponding to the matched string.  This establishes the proof.
\end{proof}

\begin{theorem}\label{T: null_1}
For $n\in \mathbb{N}$ the number $N(n,1)$ of Toeplitz matrices $\mathcal{A}_n$ with nullity $1$ is $2^{2n-2}+2^{2n-1}$.
\end{theorem}
\begin{proof}
For $n=1$ the possible nullity strings are $1,1$ and $0,1$.   By Theorem \ref{T: null_same_count} the number of Toeplitz matrices with nullity string $1,1$ is $2$.  For $n=1$ the number of Toeplitz matrices $\mathcal{A}_1$ with nullity string $0,1$ is $1$ (as $\mathcal{A}_0$ is $[1]$ and $\mathcal{A}_1$ must be $\left[ \begin{smallmatrix} 1 & 1 \\ 1 & 1 \end{smallmatrix}\right]$)
so the formula holds for $n= 1$, i.e. $N(1,1) = 3 = 2^{2\cdot 1-2}+2^{2\cdot 1 -1}$.

For $n=2$ we have nullity strings $1,1,1$ and $1,2,1$ with $5 = (2+3)2^{2-2}$ corresponding Toeplitz matrices $\mathcal{A}_2$, we have nullity string $0,1,1$ with $2$ corresponding Toeplitz matrices, and we have $1,0,1$ and $0,0,1$ with $5= \theta(1) +2\eta(1)$ corresponding matrices.  The total is then $N(2,1) = 12 = 2^2 + 2^3$ matrices.

Let $n\geq 3$.  From the Corollary above the number of Toeplitz matrices with nullity $1$ which satisfies $\nu_j >0$ for $j=0,1,\dots,n$ is $(n+3)2^{n-2}$.  Next we fix $j\in \{0,1,\dots,n-1\}$ and count the number of Toeplitz matrices with nullity $1$ for which $\nu_j=0$ and $\nu_k>0$ for $j<k\leq n$.  For the case $j=0$ then the nullity string $\nu_0,\nu_1,\dots,\nu_n$ satisfies $\nu_0 = 0, \nu_1 =1,\dots, \nu_n = 1$ and $\nu_k > 0$ for $1 \leq k \leq n$.  By the previous Corollary and the comment in the first paragraph of this proof, the number of Toeplitz matrices satisfying these conditions is $(n+2)2^{n-3}$.  Now suppose $1 \leq j \leq n-2$ then using Theorem \ref{T: 100_101} and Theorem \ref{T: zero_zero} the number of Toeplitz matrices with $\nu_j=0$ and $\nu_k > 0$ for $j+1 \leq k \leq n$ is $(\theta(j)+2\eta(j))(n-j+2)2^{n-j-3}$.  If $j=n-1$ then again by Theorem \ref{T: 100_101} and Theorem \ref{T: zero_zero} the number of Toeplitz matrices satisfying $\nu_{n-1} = 0$, $\nu_n = 1$ is $\theta(n-1)+2\eta(n-1)$.  Then we get
that the number $N(n,1)$ of $\mathcal{A}_n$'s with nullity $1$ for $n\geq 3$ is
\[ (n+3)2^{n-2}+(n+2)2^{n-3} + \sum_{j=1}^{n-2}(\theta(j)+2\eta(j))(n-j+2)2^{n-j-3} + (\theta(n-1)+2\eta(n-1)). \]

From the formulas for $\theta$ and $\eta$ in Theorem \ref{T: count invertible} we have $\theta(j) + 2\eta(j) = (2^{2j+2}-1)/3$ so for $n=3$ the expression above for $N(3,1)$ is

\[ 6\cdot 2 + 5 + (\theta(1) + 2\eta(1))\cdot 4\cdot 2^{-1} + (\theta(2)+2\eta(2)) = 12 + 5 + 5\cdot 2 + 21 = 48 = 2^{4}+2^5.\]

Now suppose for some $n\geq 3$ that $N(n,1) = 2^{2n-2} + 2^{2n-1}$.  Replacing $n$ with $n+1$ in the expression for $N(n,1)$ above, giving the expression for $N(n+1,1)$,
we obtain the following expression for $N(n+1,1)-2N(n,1)$:
\[N(n+1,1)-2N(n,1) = 2^{n-1} + 2^{n-2} + \sum_{j-1}^{n-2}(\theta(j)+2\eta(j))2^{n-j-2} + (\theta(n) + 2\eta(n)),\] or
\[N(n+1,1)-2N(n,1) = 2^{n-1} + 2^{n-2} + \sum_{j=1}^{n-2} \frac{2^{2j+2}-1}{3}2^{n-j-2} + \frac{2^{2n+2}-1}{3}. \]  A routine calculation shows that this expression is equal to $2^{2n-1} + 2^{2n}$.  Therefore,
\[N(n+1,1) = 2N(n,1) +  2^{2n-1} + 2^{2n} = 2(2^{2n-2}+2^{2n-1}) = 4(2^{2n-2}+2^{2n-1}) = 2^{2n}+2^{2n+1}.\]
This establishes the induction step and the result is proved.

\end{proof}

\begin{definition}\label{D: P_T}  We use the notation $P(m,k)$ and $T(n,k)$ to denote the following.
\begin{enumerate}
\item For $m\in \mathbb{N}\cup \{0\}$ and $k\in \mathbb{N}$ let $P(m,k)$ denote the number of nullity strings $\nu_0,\nu_1,\dots,\nu_m$ for which $\nu_j >0$, for $j=0,\dots,m$, and $\nu_m = k$.
\item For $n\in \mathbb{N}\cup \{0\}$ and $k\in \mathbb{N}$ let $T(n,k)$ denote the number of Toeplitz matrices $\mathcal{A}_n$ for which $\nu_n = k$.
\end{enumerate}
\end{definition}

\begin{remark}\label{R: P_prelims} The following observations will be useful.
\begin{enumerate}
\item  The only nullity string $\nu_0,\dots,\nu_{k-1}$ of length $k$ ending in $k$ is $1,2,\dots,k$, and there is only one Toeplitz matrix $\mathcal{A}_{k-1}$ with this string, so $P(k-1,k)$ = 1.  Clearly $P(m,k) = 0$ for $m < k-1$.
\item  From Theorem \ref{T: nonzero_string_end_s} above we have $P(m,1) = P(m+k-1,k)$ for all positive integers $k$.
\end{enumerate}
\end{remark}
We now show, in fact, that $T(n+k-1,k) = T(n,1)$ for all $m \geq 0$ and positive integers $k$. Recall from Theorem \ref{T: null_1} that $T(n,1)$ was calculated by dividing the set of all $\mathcal{A}_n$'s with nullity $1$ into disjoint sets $S^{(n)},S_0^{(n)},\dots,S_{n-1}^{(n)}$ where $S^{(n)}$ is the set of $\mathcal{A}_n$'s whose string $\nu_0,\nu_1,\dots,\nu_{n}$ has no zeroes, $S_0^{(n)}$ consists of those $\mathcal{A}_n$'s for which $\nu_0 = 0$ and $\nu_j > 0$ for $j=1,2,\dots,n$,
$S_1^{(n)}$ consists of those $\mathcal{A}_n$'s for which $\nu_1 = 0$ and $\nu_j > 0$ for $j=2,\dots,n,\,\,$and so on, up through $S_{n-1}^{(n)}$ the set of all of those $\mathcal{A}_n$'s for which $\nu_{n-1} = 0$ and $\nu_n=1$. Then by the proof of Theorem \ref{T: null_1}, $\vert S^{(n)} \vert = P(n,1), \vert S_0^{(n)} \vert = P(n-1,1), \vert S_j^{(n)} \vert = (\theta(j)+2\eta(j))P(n-j-1,1)$ for $j=1,\dots,n-2$ and $\vert S_{n-1}^{(n)} \vert = \theta(n-1)+2\eta(n-1)$.  Similarly the set of all Toeplitz matrices $\mathcal{A}_{n+k-1}$ with nullity $k$ splits into disjoint sets $S^{(n+k-1)},S_0^{(n+k-1)},\dots,S_{n-1}^{(n+k-1)}$ where $S^{(n+k-1)}$ consists of all those Toeplitz matrices $\mathcal{A}_{n+k-1}$ for which $\nu_0,\dots,\nu_{n+k-1}$ has no zeroes and $\nu_{n+k-1}=k$, and for $j=0,1,\dots,n-1, S_j^{(n+k-1)}$ consists of those for which $\nu_j = 0$ while $\nu_{j+1}$ through $\nu_{n+k-1}$ are nonzero and $\nu_{n+k-1}=k$.  By Theorem \ref{T: nonzero_string_end_s}, and arguing as in the proof of Theorem \ref{T: null_1}, we deduce that
$\vert S^{n+k-1} \vert = P(n+k-1,k), \vert S_0^{(n+k-1)} \vert = P(n+k-2,k), \vert S_j^{(n+k-1)}\vert = (\theta(j)+2\eta(j))P(n-j+k-2,k)$ for $j=1,\dots,n-2$, and $\vert S_{n-1}^{(n+k-1)} \vert = \theta(n-1)+2\eta(n-1)$.  Since $P(r,1) = P(r+k-1,k)$ by the remark above, however, and $\vert S^{(n)} \vert + \sum_{j=0}^{n-1} \vert S_j^{(n)}\vert $ is the number of $\mathcal{A}_n$'s with nullity $1$ (resp., $\vert S^{(n+k-1)} \vert + \sum_{j=0}^{n-1} S_j^{n+k-1)}$ is the number of matrices $\mathcal{A}_{n+k-1}$ with nullity $k$, we have shown the following.
\begin{theorem}\label{T: T_count}
For any positive integer $k$ and any non-negative integer $n$, $T(n+k-1,k) = T(n,1)$.
\end{theorem}

We can combine the two preceding results to obtain the following.

\begin{corollary}  The number of $(n+1)\times (n+1)$ Toeplitz matrices $\mathcal{A}_n$ with nullity $1$ is $2^{2n-2} + 2^{2n-1}$, with nullity $2$ is $2^{2n-4}+2^{2n-3}$, and so on.  The number with nullity $n$ is $3$ and the number with nullity $n+1$ is $1$ (the zero matrix).
\end{corollary}

\begin{remark}\label{R: Daykin} With this corollary and Theorem \ref{T: count invertible} we have recovered a result of D. Daykin \cite{Day} on the enumeration of Toeplitz matrices over $GF(2)$ of specified rank.  (In fact, Daykin has made these computations over any finite field, see also \cite{KalL}).
\end{remark}

\section*{Acknowledgments}\label{S: acknowl}
The authors are grateful to the referee for helpful comments and for bringing references \cite{ER13}, \cite{HR84} and \cite{HR02} to our attention.

\end{document}